\documentclass{article}
\usepackage[utf8]{inputenc}
\usepackage{comment}
\usepackage{amsfonts,amsmath,amssymb,amsthm}
\usepackage{xcolor}
\usepackage{hyperref} 
\hypersetup{
    colorlinks=false, 
    linktoc=all 
}
\usepackage{mathtools} 

\newtheorem{theorem}{Theorem}[section] 
\numberwithin{theorem}{section} 

\newtheorem{lemma}[theorem]{Lemma}
\newtheorem{proposition}[theorem]{Proposition}

\newtheorem{definition}{Definition}

\newcommand{\ep}{\varepsilon}

\DeclareMathOperator{\sgn}{sgn}
\DeclareMathOperator{\supp}{supp}
\DeclareMathOperator{\R}{\mathbb{R}}

\DeclareMathOperator{\Tr}{Tr}

\newcommand{\F}{\mathcal{F}}

\usepackage{authblk} 
\title{Free Moment Measures and Laws}
\author[1]{Juniper Bahr\thanks{jbahr@math.ucla.edu}}
\author[1]{Nick Boschert \thanks{nickboschert@math.ucla.edu}}
\affil[1]{University of California, Los Angeles}
\date{\today}

\begin{document}

\maketitle

\begin{abstract}
In \cite{klartag}, it was shown that convex, almost everywhere continuous functions coordinatize a broad class of probability measures on $\mathbb{R}^{n}$ by the map $U\mapsto \left(\nabla U\right)_{\#}e^{-U}dx$. We consider whether there is a similar coordinatization of non-commutative probability spaces, with the Gibbs measure $e^{-U}dx$ replaced by the corresponding free Gibbs law.  We call laws parameterized in this way free moment laws. We first consider the case of a single (and thus commutative) random variable and then the regime of $n$ non-commutative random variables which are perturbations of freely independent semi-circular variables. We prove that free moment laws exist with little restriction for the one dimensional case, and for small even perturbations of free semi-circle laws in the general case.
\end{abstract}

\textbf{\textit{Keywords---}}Free Probability, Optimal Transport, Operator Algebra, Brenier Map, Free Gibbs Law

\section{Introduction}

Fix a measure $\mu$ on $\mathbb{R}^{n}$; following \cite{klartag}, we say that $\mu$ is a \emph{moment measure} with \emph{potential $u$} when $u$ is a convex function satisfying $\mu = (\nabla u)_{\#} \rho$ and $\rho$ is the Gibbs measure $\frac{1}{Z} e^{-u} dx$.  We also say $\mu$ is the moment measure of $u$.  Cordero-Erausquin and Klartag in \cite{klartag} show that a finite Borel measure $\mu$ is a moment measure with some convex essentially continuous potential $u$ if and only if $\mu$ has barycenter zero (in particular, a finite first moment) and is not supported in a lower dimensional hyperplane. This result is proven variationally, although we will rely more directly on another variational approach taken in \cite{santambrogio} which is more closely related to optimal transport.  In Section 2 we describe a functional in terms of $\mu$ considered in \cite{santambrogio} whose optimizer is $\rho = e^{-u} dx$, as well as this functional's analog in free probability.

Voiculescu introduced free probability theory in \cite{Voiog}. He later introduced the notion of free entropy in a series of papers \cite{Voi1,Voi2,Voi3,Voi4,Voi5}; see also \cite{Voiculescu} for a summary.  The setting for free probability is that of non-commutative (nc) probability spaces---pairs $(M,\tau)$, where $M$ is a $*$-algebra (often a $C^*$ or $W^*$ algebra) and $\tau$ is a state, a functional which is both positive ($\tau(x^*x) \ge 0$) and satisfies $\tau(1) = 1$.  In this paper we will further assume our state $\tau$ is a trace, i.e., $\tau(ab) = \tau(ba)$.  The analogy to classical probability spaces $(\Omega, \mathcal{F}, P)$ is made by interpreting $M$ as the space of $\mathcal{F}$-measurable essentially bounded functions on $\Omega$, and $\tau$ as the expectation on this space with respect to $P$.

Consistent with this analogy, a nc random variable is an element of $M$.  Similarly, a vector valued nc random variable is an \(n\)-tuple $(X_{1},...,X_{n})$ of elements of $M$.  Note that when $M$ is a $C^{*}$ or $W^{*}$ algebra this can be slightly more restrictive than the classical notion, since we assume that these random variables have bounded norm, corresponding classically to an almost surely bounded random variable.  The linear map sending non-commutative polynomials $P$ to $\tau(P(X_1, \dots, X_n))$ is the \emph{law} of these random variables.

It is thus natural to ask if moment measures have an analog in free probability.  This is especially of interest to us because moment measures $\mu$ are in a sense parametrized by their potentials $u$.  Of course there is a natural way of doing this in $\mathbb{R}^n$, considering the density with respect to the Lebesgue measure. However, in the free case, the notion of density is ill-defined.

There is an analog of Gibbs measures $\tfrac{1}{Z} e^{-u} dx$ to free probability: free Gibbs laws (see \cite{BnS, Voiculescu}).  Where Gibbs laws minimize
\begin{align*}
    \mathcal{E}(\mu)+\int u d\mu,
\end{align*}
with $\mathcal{E}$ is the classical entropy ($\mathcal{E}(f\, dx) = \int f \log f \, dx$), free Gibbs laws minimize
\begin{align*}
    -\chi(\tau) + \tau(U),
\end{align*}
where $\chi$ is free entropy, first defined by Voiculescu (see the survey paper \cite{Voiculescu} for more information). Here $U$, which is assumed to be self-adjoint, is the potential for the free Gibbs law $\tau$.  
\begin{definition}[\cite{Voiculescu,Gui06,GMS}]
The free Gibbs law $\tau_U$ associated to the potential $U$ is the minimizer of $-\chi(\tau) + \tau(U)$ if it exists.
\end{definition}
There are two cases when such laws are known to exist. The first is when $U$ is a n.c. power series which is a small perturbation of quadratic (see \cite{GMS}). 

The second is in the single variable case when $U$ is bounded below, satisfies a growth condition, and satisfies a locally H\"older condtinuous-like condition (see \cite[Remark 3]{dM}) where we also get uniqueness.  In this latter case, the free entropy is the negative of log energy, ${\iint \log|s-t|d\mu(s)d\mu(t)}$, (see \cite{Voiculescu}).
The above optimization implies (and by \cite{GMS}, for $U$ which are small perturbations of quadratic the above, is equivalent to) the integration by parts formula or Schwinger-Dyson (type) equation:
\begin{align*}
\tau(P\cdot \mathcal{D}U)=\tau\otimes \tau\otimes \Tr(JP),
\end{align*}
where $U\in \mathbb{C}\langle X_{1},...,X_{n}\rangle$ is the potential of the law which is assumed to be self-adjoint, and $P$ is an arbitrary $n$-tuple of nc polynomials in $X$. Letting $M=W^{*}(X_{1},...,X_{n})$, we have that Voiculescu's cyclic gradient $\mathcal{D}=(\mathcal{D}_{x_{1}},...,\mathcal{D}_{x_{n}})$, the difference quotient derivative $\partial = (\partial_{x_{1}}, \dots, \partial_{x_n})$, and the (difference quotient) Jacobian $J$ are linear maps on the following spaces
\begin{align*}
\mathcal{D}_{x_i} &: M \to M
\\
\partial_{x_i} &: M\to M\otimes M^{op}
\\
J &: M^{n}\to M_{n \times n}\left(M\otimes M^{op}\right),
\end{align*}
defined by
\begin{align*}
\mathcal{D}_{x_{i}} (x_{i_{1}} \cdots x_{i_{n}}) 
    &= \sum_{j=1}^{n} \delta_{i,i_{j}} x_{i_{j+1}}
        \cdots x_{i_{n}} x_{i_{1}} \cdots x_{i_{j-1}}
\\
\partial_{x_{i}} (x_{i_{1}} \cdots x_{i_{n}})
    &= \sum_{j=1}^{n} \delta_{i,i_{j}}  x_{i_{1}} \cdots x_{i_{j-1}}
        \otimes x_{i_{j+1}} \cdots x_{i_{n}}.
\\
(JP)_{ij} 
    &= \partial_{x_j} P_{i}
\end{align*}

The above Schwinger-Dyson equation is the nc analog of 
\begin{align*}
    \mathbb{E} (f \cdot \nabla U) = \mathbb{E}(\Tr (\mathrm{Jac} f))
\end{align*}
which holds for log concave Gibbs laws $\frac{1}{Z} e^{-U} \, dx$, where $\mathrm{Jac}$ is the classical Jacobian. These free Gibbs laws are known to exist in the multi-variable case when $U$ is a small perturbation of the semi-circle potential (\cite{GMS}). In the single variable case, this can be relaxed to ordinary convexity along with growth conditions: $U(x)$ must go to infinity as $|x|$ does (and thus must grow at least as $|x|$).

We then define free moment laws as follows
\begin{definition}
The law $\tau$ of the nc random variables $X_1, \dots, X_n$ is a free moment law if there exists a self-adjoint nc power series $U$ such that the free Gibbs law $\tau_U$ is well defined and is the law of nc random variables $Y_1, \dots, Y_n$ such that
\begin{align*}
    (X_1, \dots, X_n) = (\mathcal{D} U)(Y_1, \dots, Y_n)
\end{align*}
\end{definition}

In the single variable case, laws have corresponding measures, and so we will discuss free moment measures instead of free moment laws.

Our main result is to show that certain free Gibbs laws are in fact free moment laws.

We organize the paper as follows.  In Section 2, we discuss the single variable case where we prove the most general existence result for free moment measures using a variational approach.   We will also provide a few examples and contrast them with the classical case. In Section 3, we discuss the existence of free moment laws for a certain class of free Gibbs laws which are close to the semicircular law.  We proceed in this case by a contraction mapping argument.  

\subsection*{Acknowledgements}

Research supported by NSF grant DMS-1762360.  We would like to thank Dimitri Shlyakhtenko for insightful conversations and advice, as well as Max Fathi for introducing Dimitri (and hence us) to the concept.

\section{The Single Variable Case}

\subsection{Main Result for One Variable}

In the case of a single (non-commutative) random variable $X$ in the nc probability space $(M, \tau)$, the law of $X$ can
be given as a functional on the space of polynomials in a single variable by letting the law $\tau_X(p)$ for a polynomial $p(z)$ be $\tau_X(p) = \tau(p(X))$. Alternatively we can view the law of $X$ as a probability measure $\mu$, using positivity and the Riesz-Markov theorem. 

Suppose $X$ has law $\tau$ with corresponding measure $\mu$.  Then if $\tau$ is a free moment law, there exists $Y$ with law $\tau_u$ such that $X = (\mathcal{D} u)(Y)$.  As the cyclic gradient of a function in one variable is equal to the ordinary derivative, the pushforward condition is equivalently $X = (u')(Y)$.  If $\mu$ is the measure corresponding to $\tau$ and $\nu_u$ is the measure corresponding to $\tau_u$, then we have $\mu = (u')_{\#} \nu_u$.

We refer to the measures associated to free Gibbs laws in one dimension as free Gibbs measures. The authors emphasize that the idea of free Gibbs measures is not wholly novel; Indeed, free Gibbs laws were defined earlier (see Def 1), it was known (see \cite{Voiculescu},\cite{Voi2}) that $\chi(\tau)$ reduces in the single variable case to log energy, and minimizers of $-\chi(\tau)+\tau(U)$ have already been studied, e.g. in (\cite{dM}).
\begin{definition}
The free Gibbs measure $\nu_u$ associated to the convex function $u : \R \to \R$ is the measure corresponding to the free Gibbs law $\tau_u$, if it exists.  In other words, $\nu_u$ is the minimizer of
\begin{align*}
    \iint \log |s-t| \,d\mu(s)\,d\mu(t) + \int u(s)\,d\mu(s)
\end{align*}
if it exists.
\end{definition}
\begin{definition}
A real probability measure $\mu$ is a free moment measure if
\begin{align*}
    \mu = (u')_{\#} \nu_u
\end{align*}
for some convex function $u : \mathbb{R} \to \mathbb{R}$.
\end{definition}

Our main result in this section is Theorem \ref{main-thm-section-2} which implies that if $\mu$ is a probability measure on $\mathbb{R}$ other than $\delta_0$ with finite second moment and barycenter zero, there exists a convex $u : \mathbb{R} \to \mathbb{R}$ such that $\mu = (u')_{\#} \rho$ and $\rho = \nu_u$, where $\nu_u$ is the free Gibbs measure associated to the potential $u$.  Observe also that if $\mu$ is centered, then $u$ must have a derivative which changes signs, and so $u(x) \to \infty$ as both $x \to \pm \infty$.  Through prior understanding of free Gibbs measures, we'll also have that $\rho$ is absolutely continuous with respect to Lebesgue measure and $2 \pi H(\rho) (x) = u'(x)$ for any $x \in \supp (\rho)$.  Here $H \rho$ is the Hilbert transform of $\rho$, given by the principal value integral
\begin{align*}
    H \rho(t) = \frac{1}{\pi} \textrm{ PV\!\!} \int_{\mathbb{R}} \frac{1}{t-x} \, d\rho(x).
\end{align*}

For a brief computational guide to solving $2 \pi (H \rho)(x) = u'(x)$ for $x \in \supp{\rho}$ for a fixed $u$, see the appendix.  See also \cite{dM} for more examples.




\subsection{The Functional \texorpdfstring{$\mathcal{F}(\rho)$}{F(ρ)}}

In the classical case of moment measures, we
are searching for $\rho = \tfrac{1}{Z} e^{-u} \, dx$, the log concave Gibbs measure with
real convex potential $u$ satisfying $(\nabla u)_{\#} \rho = \mu$ for some $\mu$.  Here $Z$
is the constant that makes $\rho$ a probability measure.

It is possible to find such $\rho$ when $\mu$ has barycenter zero and is not supported on a hyperplane (which for $\mathbb{R}^1$ only means it isn't $\delta_0$) \cite{klartag}.  The measure $\rho$ can be found by considering the functional
\begin{align*}
  \int \rho \log \rho \, dx + \tfrac{1}{2} \int x^2 \rho(x) \, dx + \tfrac{1}{2} \int x^2 \, d\mu- \tfrac{1}{2} W_2^2(\rho, \mu) &= \int \rho \log \rho \, dx + T(\rho, \mu)
  \\
                                                                                               &\eqqcolon \mathcal{E}(\rho) + T(\rho,\mu)
\end{align*}
where $W_2$ is the Wasserstein distance between $\rho$ and $\mu$,
$T(\rho,\mu)$ is the maximal correlation functional defined as follows and $\mathcal{E}$ is the negative differential entropy, $\mathcal{E}(\rho \,dx) = \int \rho \log \rho \,dx$.

The measure $\rho$ satisfying $(\nabla u)_{\#} \rho = \mu$ and $\rho = \tfrac{1}{Z} e^{-u} \,dx$ is then the minimizer of $\mathcal{E}(\rho) + T(\rho,\mu)$ when such a $\rho$ exists \cite{santambrogio}.

\begin{definition}[\cite{santambrogio}]
The maximal correlation functional $T(\rho,\mu)$ is given by
\begin{align*}
    T(\rho,\mu) &= \sup \left\{\int x \cdot y \,d\gamma \mid \gamma \in \Pi(\rho,\mu)\right\}
    \\
    &= \frac{1}{2} \int x^2 \,d\rho + \frac{1}{2} \int x^2 \,d\mu - \frac{1}{2} W_2^2(\rho,\mu)
\end{align*}
where $\Pi(\rho,\mu)$ is the set of transport plans, i.e. probability measures on $\mathbb{R}^n \times \mathbb{R}^n$ with marginals $\rho$ and $\mu$.
\end{definition}

We replace the entropy term of $\mathcal{E}(\rho) + T(\rho,\mu)$ with free entropy, which in the 1-D case is the log energy (up to a constant) \cite{Voi1}:
\begin{align*}
  L(\rho) = \iint - \log |s-t| \, d\rho(s) \,d \rho(t).
\end{align*}

This is justified by the following proposition:
\begin{proposition}[\cite{santambrogio} p.\ 14]
  For $V$ convex, the minimizer of the functional
  \begin{align*}
    \mathcal{E}(\rho) + \int V \rho \,dx = \int \rho \log \rho + V \rho \,dx
  \end{align*}
  over $\rho$ probability measures with finite second moment is the density of the Gibbs measure $\rho = \tfrac{1}{Z} e^{-V}$. 
\end{proposition}

As free entropy in the 1-D case is log energy up to a  constant, we recall that the minimizer of the functional
\begin{align*}
    L(\rho) + \int V \,d\rho
\end{align*}
is the free Gibbs measure $\nu_V$ if it exists.  Thus we see how $\mathcal{E}$ and $L$ play analogous roles for Gibbs measures and free Gibbs measures.


Following this analogy, we define the following functional:
\begin{align}
  \mathcal{F}(\rho) = L(\rho) + T(\rho,\mu). \label{1d-std-functional}
\end{align}

\subsection{Sufficiency of Minimizing \texorpdfstring{$\mathcal{F}(\rho)$}{F(ρ)}}

 Throughout this section, $\rho$ will be assumed to have finite second moment unless otherwise specified.

Following \cite{santambrogio}, we can rewrite $T(\rho,\mu)$ a few ways.  First, we use the maximal correlation formulation:
\begin{align*}
  T(\rho,\mu) = \sup \left\{\int x \cdot y \,d \gamma(x,y) \;\middle|\; \gamma \in \Pi(\rho,\mu)\right\}
\end{align*}
where $\Pi(\rho,\mu) = \{ \gamma \in \mathcal{P}(\mathbb{R} \times \mathbb{R}) \mid (\pi_x)_{\#} \gamma = \rho, (\pi_y)_{\#} \gamma = \mu\}$ is the space of measures with marginals $\rho$ and $\mu$.  Here $\mathcal{P}(X)$ denotes the space of probability measures on $X$.

This maximization problem has an equivalent dual problem, a minimization with the same optimal value:
\begin{align*}
  T(\rho,\mu) = \min \left\{\int u \,d\rho + \int u^* \,d\mu \;\middle|\; u \text{ convex, lower semicontinuous}\right\}.
\end{align*}

This lets us rewrite (\ref{1d-std-functional}) as
\begin{align}
  \mathcal{F}(\rho) = \min \left\{ \underbrace{\iint - \log |s-t| \,d\rho(s) \,d\rho(t) + \int u \,d\rho + \int u^* \,d\mu}_{\mathcal{G}(\rho,u)} \right\} \label{1d-general-functional}
\end{align}
minimizing over the set where $\rho \in \mathcal{P}(\mathbb{R})$, $\mathbb{E}_{\rho}(|x|)<\infty$, and $u$ is convex and lower semicontinuous.  Here $u^*$ denotes the Legendre transform
\begin{align*}
  u^*(y) = \sup_x \left(x\cdot y - u(x)\right).
\end{align*}  
We'll define $\mathcal{G}(\rho,u) = \iint - \log |s-t| \,d\rho(s) \,d\rho(t) + \int u \,d\rho + \int u^* \,d\mu$ and so $\mathcal{F}(\rho) = \min_u \mathcal{G}(\rho,u)$.

By minimizing $\mathcal{G}(\rho,u)$ first in $u$ for each $\rho$, we can appeal to Santambrogio's analysis of the maximal correlation functional and deduce that $(u')_{\#} \rho = \mu$ \cite{santambrogio}.  Next, for optimal $u$, minimizing in $\rho$ lets us rely on \cite{dM} to see that $\rho = \nu_u$, the free Gibbs measure associated to $u$.  This is explained in further detail in Theorem \ref{main-thm-section-2}.

\subsection{Minimizing the functional}

We now adapt the proof from \cite{santambrogio} to show that $\mathcal{F}$ has a minimizer.  First we prove weak lower semicontinuity of the $L(\rho)$ term and show that it's bounded below by an expression involving the first moment of $\rho$, a bound we will combine with a known bound on $T(\rho,\mu)$.  We then prove a kind of convexity of $L(\rho)$ in the Wasserstein space $\mathbb{W}_2$.  We use this to deduce the existence and uniqueness of the minimizer of $\mathcal{F}(\rho) = L(\rho) + T(\rho,\mu)$.

\begin{lemma} \label{first-moment-bound} Assume that $\rho$ is a probability measure with finite first moment.  Then the log energy $L(\rho)$ satisfies the bound $L(\rho) \ge -\sqrt{2\int |s| \,d\rho(s)}$.

Furthermore, when $\rho_n$ and $\rho$ are probability measures with $\rho_n \rightharpoonup \rho$ weakly and $\int |x| \,d\rho_n \le C$ for some $C > 0$ and all $n \in \mathbb{N}$, then $L(\rho) \le \liminf_{n \to \infty} L(\rho_n)$.  In short, weak lower semi-continuity of $L$ if the first moments are uniformly bounded. 

\begin{proof} To bound $L(\rho)$, we split it into three terms with a method inspired by \cite{santambrogio}.  In that paper, Santambrogio splits up the integrand of the entropy term into three parts using a Legendre transform of $x \log x$ for a key inequality.

We need an analogous inequality:
\begin{align*}
    -1 + \log \left(\frac{1}{h}\right) - \log |x| \ge -|x| h
\end{align*}
for any $x \neq 0$ and $y > 0$.  This inequality can be derived from the Legendre transform of $- \log x$, the analogous term in our case, but it is more easily derived from an application of $1 + \log a \le a$ where $a = |x| h$.

With this inequality, we consider the decomposition:
\begin{align*}
    L(\rho) &= \iint -1 + \log \left(\frac{1}{h}\right) - \log |s-t| + h |s-t| \,d\rho(s) \,d\rho(t)
    \\
    &\quad \iint -\log \left(\frac{1}{h}\right) \,d\rho(s)\,d\rho(t) + \iint 1 - |s-t| h \,d\rho(s) \,d\rho(t)
    \\
    &= \mathrm{I} + \mathrm{II} + \mathrm{III}.
\end{align*}
While this decomposition holds regardless of $h > 0$, we'll select $h$ inspired by the proof in \cite{santambrogio}.  We choose
\begin{align*}
    h(s,t) &= e^{-\sqrt{|s-t|}}.
\end{align*}

Observe that term (I) has a positive integrand by the inequality mentioned above.  Since the integrand is continuous and bounded below, we have that (I) is lower semi-continuous with respect to weak convergence of measures.

Next, we bound the second term
\begin{align*}
    II &= \iint - \log \left(\frac{1}{e^{-\sqrt{|s-t|}}}\right) \,d\rho(s)\,d\rho(t) = \iint - \sqrt{|s-t|} \,d\rho(s)\,d\rho(t)
    \\
    &\ge - \sqrt{\iint |s-t| \,d\rho(s) \,d\rho(t)}
    \\
    &\ge -\sqrt{\iint |s| + |t| \,d\rho(s) \,d\rho(t)} = - \sqrt{2 \int |s| \,d\rho(s)}
\end{align*}
where the first inequality follows by Cauchy-Schwarz and the fact that $\rho$ is a probability measure.

Note that $\sqrt{x}/|x| \to 0$ as $x \to \infty$.  We'll use this to show that (II) is weakly lower semi-continuous for $\rho_n$ having bounded first moments.  

Observe that as $\int |x| \,d\rho_n \le C$, we have
\begin{align*}
    \left| \int_{[-M,M]^c} -\sqrt{x} \,d\rho_n \right| \le \frac{\sqrt{M}}{M} \int_{[-M,M]^c} |x| \,d\rho_n \le \frac{C}{\sqrt{M}}
\end{align*}
for any $M>1$.  Fix $\ep > 0$.  Thus we may choose $M$ so large that $\left| \int_{[-M,M]^c} -\sqrt{x} \,d\rho_n\right| < \ep$.  We now write
\begin{align*}
    \iint - \sqrt{|s-t|} \,d\rho_n(s)\,d\rho_n(t) &= \iint_{|s-t|>M} - \sqrt{|s-t|}\,d\rho_n(s)\,d\rho_n(t)
    \\&\quad + \iint - \sqrt{|s-t|} \chi_{|s-t|\le M} \,d\rho_n(s) \,d\rho_n(t).
\end{align*}
The first term is bounded in absolute value by $\ep$.  As the second term is integration against a lower semi-continuous functions which is bounded from below, it is a lower semi-continuous function with respect to weak convergence of measures.  

Combining these facts,
\begin{align*}
    \iint - \sqrt{|s-t|} \,d\rho(s)\,d\rho(t) &\le \liminf_{n \to \infty} \iint - \sqrt{|s-t|} \,d\rho_n(s) \,d\rho_n(t) + 2 \ep
\end{align*}
for any $\ep > 0$ and thus we have the desired weak lower semi-continuity of this term.

Finally we write
\begin{align*}
    III = \iint 1 - |s-t| e^{-\sqrt{|s-t|}} \,d\rho(s)\,d\rho(t)
\end{align*}
and observe that the integrand is bounded between 0 and 1, so $0 \le III \le 1$.  The integrand being continuous and bounded implies that this term is continuous with respect to the weak convergence of measures.

Combining these inequalities, we have
\begin{align*}
    L(\rho) = I + II + III \ge 0 - \sqrt{2\int|s|\,d\rho(s)} + 0
\end{align*}
as desired.

Furthermore, we have the desired weak lower semi-continuity in each term, and so it holds that $L(\rho) \le \liminf_{n \to \infty} L(\rho_n)$ when the $\rho_n$ all have bounded first moments.
\end{proof}

\end{lemma}

We will need another lemma to obtain uniqueness of the minimizer.  We'll show that $L(\rho)$ is displacement convex, i.e., convex along geodesics in the Wasserstein space $\mathbb{W}_2$.
\begin{lemma}
\label{L-disp-convex}
The functional $L(\rho)$ is displacement convex.  Specifically, if $\rho_t$ is any geodesic connecting $\rho_0$ to $\rho_1$ in the Wasserstein space $\mathbb{W}_2$, then $L(\rho_t)$ is convex.

Furthermore, $L$ is strictly displacement convex for measures which are not translates.  That is, if $\rho_0$ and $\rho_1$ are not translates of each other, by which we mean one is not the pushforward of the other under a map of the form $x \mapsto x+c$, then $L(\rho_t) < (1-t)L(\rho_0) + t  L(\rho_1)$.
\end{lemma}
\begin{proof} Let $\rho_{0}$ and $\rho_{1}$ to be two measures with finite second moments (so that they're in $\mathbb{W}_2$).  Then let $\gamma$ be the optimal transport plan between them (see \cite{transport} or \cite{villani} for a thorough introduction to these ideas), and consider $\rho_t = \pi_{t\#} (\gamma)$ where $\pi_t(x,y) = (1-t)x+ty$.  Note that $\rho_t$ is the geodesic connecting $\rho_0$ and $\rho_1$ in $\mathbb{W}_2$, and all geodesics have this form \cite[Chap.\ 5]{transport}.  We then observe
\begin{align*}
L(\rho_{t}) &= \iint - \log|s-r| \,d\rho_{t}(s) \,d\rho_{t}(r)
\\
 &= \iint -\log\big|(1-t)x+ty-(1-t)x'-ty'\big| \,d\gamma(x,y)\,d\gamma(x',y')
 \\
 &= \iint - \log \big|t (y-y') + (1-t)(x-x')\big| \,d\gamma(x,y)\,d\gamma(x',y')
\end{align*}
By the convexity of $- \log$, the integrand is strictly less than
\[-\left((1-t)\log|x-x'|+t\log|y-y'|\right)\] unless $x-y=x'-y'$.  Thus $L(\rho_t)$ is strictly less
than $(1-t)L(\rho_{0})+tL(\rho_{1})$ unless $\gamma$ is supported on a
translate of the diagonal, which can only occur if $\rho_{0}$ and $\rho_{1}$ are
translates of one another. 
\end{proof}

We aim to minimize $\mathcal{F}$, but we need to show now that the minimizer will have finite second moment.

\begin{proposition} \label{free-gibbs-cpt-supp}
Let $u : \mathbb{R} \to \mathbb{R}$ be convex and have a minimum so that $u(x) \ge a|x| + b$ for some $a > 0$ and real $b$.

Suppose $\rho$ is the free Gibbs measure associated to $u$ and has finite first moment.  Then $\rho$ is compactly supported and absolutely continuous with respect to Lebesgue measure.  Furthermore, $2\pi H\rho = u'$ on the support of $\rho$.
\end{proposition}

\begin{proof}
By \cite[Remark 3]{dM} and noting that the function $u$ satisfies their condition (1.2),  Theorem 1 of \cite{dM} guarantees that $\rho$ is absolutely continuous with respect to Lebesgue measure and that the support of $\rho$ is contained in the set of points such that
\begin{align*}
    h(x) = \int - \log |x-y| \,d\rho(y) + u(x)
\end{align*}
is minimal.  We can also see this by taking a first variation of the functional $\iint -\log|s-t| \,d\rho(s)\,d\rho(t) + \int u(t) \,d\rho(t)$ and considering the optimality conditions.  Theorem 1 of that paper also guarantees that $2 \pi H \rho = u'$ on the support of $\rho$, noting that $\beta = 2$ for our case in \cite[Eqn.\ 1.17]{dM}, although using absolute continuity we could also get this by considering optimality conditions for the functional defining $\nu_u$ and differentiating under the integral.

Since $U(x) \ge a|x| + b$, $-\log$ is non-increasing, and $z\mapsto \log(1+z)$ is subadditive on the positive reals, we have
\begin{align*}
    h(x) &\ge \int - \log |x-y| \,d\rho(y) + a|x| + b
    \\
    &\ge \int -\log (|x|+|y|+1) \,d\rho(y) + a|x| + b
    \\
    &\ge \int -\log(|x|+1) \,d\rho(y) + \int -\log(|y|+1) \,d\rho(y) + a|x| + b
    \\
    &\ge -\log(|x|+1) + \int - \log (|y|+1)\,d\rho(y) + a|x| + b.
\end{align*}
Note that the finite first moment of $\rho$ implies $\int -\log(|y|+1) \,d\rho(y) > -\infty$, since $\log$ has sublinear growth at $\infty$.  Thus $h(x) \to \infty$ as $x \to \infty$ or $x \to -\infty$.  Note that $h(x)$ isn't constantly $\infty$ as its integral gives the functional minimized by $\rho$.  Therefore the set where $h$ is its minimum value is compact, so $\supp(\rho)$ is compact.
\end{proof}

We now show the existence of a minimizer of $\mathcal{F}$ and prove the main theorem of this section.
\begin{theorem} \label{main-thm-section-2}
Let $\mu \neq \delta_0$ be a probability measure with finite second moment.  The functional $\mathcal{F}(\rho) = L(\rho) + T(\rho,\mu)$ has a minimizer in $\mathcal{P}_2$, the space of probability measures with finite second moment, which is unique up to translation, i.e., unique up to a pushforward by the map $x \mapsto x+c$.  

The minimizer $\hat{\rho}$ is also absolutely continuous with respect to Lebesgue measure, has compact support, and satisfies $2\pi H \hat{\rho} = u'$ on its support. 

Furthermore, the following are equivalent:
\begin{enumerate}
    \item $\hat{\rho}$ is the unique centered minimizer of $\mathcal{F}(\rho)$
    \item $\hat{\rho}$ satisfies $\hat{\rho} = \nu_u$ for some convex $u$ and $(u')_{\#} \hat{\rho} = \mu$.
\end{enumerate}

\end{theorem}
\begin{proof} 
First we'll show that $\mathcal{F}$ has a minimizer unique up to translation.

Let $\rho_{n}$ be a minimizing sequence of probability measures with finite first moment.  Note that without loss of generality we may assume that the $\rho_n$ are centered, as $\mathcal{F}$ is invariant under translation.

By \cite{santambrogio}, we have that $T(\rho_n,\mu) \ge c \int |x| \,d\rho_n(x)$ for some $c > 0$ depending only on $\mu$, since $\mu$ is not supported on a hyperplane, which here means $\mu \neq \delta_0$.  Applying Lemma \ref{first-moment-bound}, we have $L(\rho_n) \ge - \sqrt{2\int |x| \,d\rho_n(x)}$.  Combining these yields a uniform bound on the first moment of the $\rho_n$, which implies the sequence is tight.  By passing to a subsequence, we can assume that $\rho_n \rightharpoonup \hat{\rho}$ weakly for some probability measure $\hat{\rho}$.  Note also that $\hat{\rho} \in \mathcal{P}_1$, the space of probability measures with finite first moment.  This is because integration against $|x|$, a lower semi-continuous function bounded from below, is a weakly lower semi-continuous functional.  

By weak convergence of $\rho_n \rightharpoonup \hat{\rho}$ and a uniform bound on the first moments, Lemma \ref{first-moment-bound} gives us that $L(\hat{\rho}) \le \liminf_{n \to \infty} L(\rho_n)$.  As we know that $T(\rho,\mu)$ is weakly lower semi-continuous in $\rho$ by \cite{santambrogio}, we have that $\hat{\rho}$ is a minimizer of $\mathcal{F}$.

We know that $\hat{\rho}$ has finite first moment, but we need to show now that it has finite second moment as well.  As part of showing this, we'll see that it must satisfy $\hat{\rho} = \nu_u$ for some convex $u$ with $(u')_{\#} \hat{\rho} = \mu$, so we'll have (1) implies (2).  Afterwards we will show uniqueness of the minimizer of $\mathcal{F}$ and then prove (2) implies (1).

Take $u$ to be a convex lower semi-continuous function which realizes the dual formulation of $T(\hat{\rho},\mu)$, that is, $T(\hat{\rho},\mu) = \int u \,d\hat{\rho} + \int u^* \,d\mu$.  Additionally, we know that $(u')_{\#} \hat{\rho} = \mu$ \cite{santambrogio}.

Simplifying $\mathcal{F}$ using $u$ now yields
\begin{align*}
    \mathcal{F}(\hat{\rho}) = \iint - \log |s-t| \,d\hat{\rho}(s) \,d\hat{\rho}(t) + \int u \,d\hat{\rho} + \int u^* \,d\mu
\end{align*}

We consider a new functional
\begin{align*}
    \mathcal{G}(\rho) = \iint - \log |s-t| \,d\rho(s) \,d\rho(t) + \int u \,d\rho + \int u^* \,d\mu
\end{align*}
and observe that since the first term is $L(\rho)$ latter two terms are larger than $T(\rho,\mu)$, we must have $\mathcal{G}(\rho) \ge \mathcal{F}(\hat{\rho})$.  Therefore $\hat{\rho}$ minimizes $\mathcal{G}$.  

However, the final term does not depend on the measure, so $\mathcal{K}(\rho) = L(\rho) + \int u \,d\rho$ is still minimized at $\hat{\rho}$.  Thus $\hat{\rho} = \nu_u$ by definition of $\nu_u$.  And as $\hat{\rho}$ has finite first moment, Proposition \ref{free-gibbs-cpt-supp} implies that $\hat{\rho}$ has compact support, and thus all its moments are finite and in particular $\hat{\rho} \in \mathcal{P}_2$.  We also get that $2 \pi H \hat{\rho} = u'$ on the support of $\hat{\rho}$.

Thus we now have that $\mathcal{F}$ has a minimizer with finite second moment, and (1) implies (2).  Let's now show that the minimizer to $\mathcal{F}$ is unique.

To show uniqueness up to translation, and thus uniqueness of a centered minimizer, we invoke the displacement convexity of both $L$ using Lemma \ref{L-disp-convex} and $T$ using \cite[Prop.\ 3.3]{santambrogio}.  Combining these will give displacement convexity of $\mathcal{F}$.  Note that displacement convexity of $T$ in \cite[Prop.\ 3.3]{santambrogio} is shown between two measures which are absolutely continuous with respect to Lebesgue measure, but the result holds just as well with no modifications when the initial measure is non-atomic and thus optimal transport maps from it still exist in the space $\mathbb{W}_2$.

Furthermore, by Lemma \ref{L-disp-convex}, we have strict displacement convexity of $L$ except between translates.  In particular, if $\rho_0$ and $\rho_1$ are minimizers and not translates of each other, then on the geodesic between them, there is some $\rho_t$ with a strictly smaller value of $L$ and a value of $T$ no larger than that of $\rho_0$ or $\rho_1$.  This is a contradiction, so any two minimizers of $\mathcal{F}$ must be translates of each other.

Finally, let's show (2) implies (1).  Let $\hat{\rho}$ satisfy $\hat{\rho} = \nu_u$ with $u$ convex and $(u')_{\#} \hat{\rho} = \mu$.  We intend to show that $\hat{\rho}$ is a minimizer of $\mathcal{F}(\rho)$, where we note that uniqueness up to translation is already guaranteed.  Also by the functional that defines $\nu_u$ not being $+\infty$, we know that $\hat{\rho}$ is non-atomic.

With $\hat{\rho}$ as above, let $\rho$ be another probability measure with finite second moment, and $f$ be the transport map between $\hat{\rho}$ and $\rho$ and let $\rho_t = (f_t)_{\#} \hat{\rho}$ where $f_t = (1-t)I + tf$.

The map $t \mapsto \mathcal{F}(\rho_t)$ is convex, so it is enough to show that its derivative at zero is non-negative.  We will compute the derivative of the log-energy term and borrow Santambrogio's calculation for $T$, which we observe does not require absolute continuity but only the existence of an optimal transport map \cite[Prop.\ 3.3]{santambrogio}.  We calculate
\begin{align*}
    \frac{d}{dt}\Big|_{t=0} L(\rho_t) &= \frac{d}{dt}\Big|_{t=0} \iint - \log |x-y| \,d\rho_t(x) \,d\rho_t(y)
    \\
    &= \iint - \frac{d}{dt}\Big|_{t=0} \log |tf(x) + x - tx - t f(y) - y + ty| \,d\hat{\rho}(x) \,d\hat{\rho}(y)
    \\
    &= - \iint \frac{f(x) - x - (f(y) - y)}{x-y} \,d\hat{\rho}(x) \,d\hat{\rho}(y)
    \\
    &= - \iint 2\frac{f(x)}{x-y} - 1 \,d\hat{\rho}(x) \,d\hat{\rho}(y)
    \\
    &= 1 - 2 \pi \int f(x) H \hat{\rho}(x) \,d\hat{\rho}(x)
    \\
    &= 1 - \int f(x) u'(x) \,d\hat{\rho}(x).
\end{align*}
The last line follows by recalling $2 \pi H \hat{\rho} = u'$ on $\supp \hat{\rho}$.

Note that for the $T$ term, we have that $\frac{d}{dt}\big|_{t=0} T(\rho_t,\mu)$ is bounded below by $\int (f(x) - x) u'(x) \,d\hat{\rho}(x)$ \cite[Prop. 3.3]{santambrogio}.  Thus combining these two terms, we find that
\begin{align*}
    \frac{d}{dt} \Big|_{t=0} \mathcal{F}(\rho_t) \ge 1 - \int x u'(x) \,d\hat{\rho}(x) \ge 0
\end{align*}
where the final inequality follows immediately from Schwinger-Dyson for $\hat{\rho} = \nu_u$ (in particular, $\tau(x u') = \tau \otimes \tau(1)$, which is an application of $2 \pi H \hat{\rho} = u'$ on the support of $\hat{\rho}$).  Thus, using the convexity of the functional and noting that the above holds for any $\rho$, we see that $\hat{\rho}$ minimizes $\mathcal{F}$. 
\end{proof}

\subsection{Examples}

We include some examples of free moment measures.

\subsubsection{Quadratic potential}

The semicircular distribution $\mu$ equals $\nu_{\tfrac{1}{2} x^2}$, so $\mu$ is a free moment measure with potential $u(x) = \tfrac{1}{2} x^2$, just as the Gaussian is a (classical) moment measure with quadratic potential. This is not surprising, as the semicircle law plays an analogous role in free probability to the Gaussian law in classical probability.

\subsubsection{Two point masses}

The next simplest example is $\mu = \tfrac{1}{2} \delta_{-1} + \tfrac{1}{2} \delta_1$ which has the potential $u(x) = \tfrac{1}{2} |x|$, since $\nu_u$ is necessarily centered when $u$ is even, and thus $(u')_{\#} \nu_u = \mu$. In this particular case, the corresponding measure is $\nu_{u}(x)=\frac{1}{\pi}\log\left|\frac{1+\sqrt{1-x^{2}}}{x}\right|$ (supported on $[-1,1]$).

\subsubsection{Quartic potential}

Given the potential $x^{4}/4$, we calculate the free Gibbs measure to be \[\nu_{u}(x)=\frac{r^{3}}{4\pi}(2x^{2}+1)\sqrt{1-\left(\frac{x}{r}\right)^{2}}dx\]
where 
\[r=\frac{2}{\sqrt[4]{3}}\]is the radius of the support. When we then push this forward by $u'=x^{3}$, we get 
\[\mu(x)=\frac{3r^{3}}{4\pi}(2+x^{-2/3})\sqrt{1-\frac{x^{2/3}}{r^{2}}}dx\] 
Thus $\mu$ is a free moment measure with potential $x^4/4$.

\subsubsection{Translation and scaling}

Note that translations $u(x+c)+d$ of a potential yield the same free moment measure as $u$ does.

Suppose $\mu$ has potential $u$ such that $(u')_{\#} \nu_u = \mu$.  Let's consider $u(x/c)$ for $c > 0$.  We'd like to find the corresponding free moment measure.  First, let's find the free Gibbs measure.

If $f(x)$ is the density for an optimizer for $\mathcal{F}_{u}(\rho)$, then $c f(cx)$ is the density for the optimizer of $\mathcal{F}_{u(cx)}(\rho)$, and vice-versa.  To see this we change variables
\begin{align*}
    \iint - \log |s-t| & c f(cs) cf(ct) \,d s \,dt + \int u(ct) c f(ct) \,dt 
    \\
    &= \iint - (\log |x-y| - \log c) f(x) f(y) \,dx\,dy+ \int u(x) f(x) \,dx
    \\
    &= \mathcal{F}_u(f(x) \,dx) + \text{constant}
\end{align*}
and note that the constant $\log c$ is irrelevant to maximization or minimization.  This tells us that if $u$ is replaced with $u(cx)$, the corresponding free Gibbs measure $\nu_u = f(x) \,dx$ is replaced with $c f(cx) \,dx$.

As a consequence, for $v(x) = u(cx)$, we have that for any $g$
\begin{align*}
    \int g(x) \,d \left(v'_{\#} \nu_v\right) &=  \int g(c u'(cx)) cf(cx) \,dx
    \\
    &= \int g(c u'(t)) \,d\nu_u(t)
\end{align*}
and so $(v')_{\#} \nu_v = c_{\#} \left((u')_{\#} \nu_u\right)$. Thus the new measure is a dilated copy of the old measure, scaled by a factor of $c$.

\section{Multivariable Case}

Instead of generalizing the variational argument, we will be applying the methods of Shlyakhtenko and Guionnet in \cite{dima}.  These methods will allow us to deal with free Gibbs laws which are near the free semicircular law (which is the free Gibbs law for the potential $\tfrac{1}{2}(X_{1}^{2}+...+X_{n}^{2})$). In order to state out main theorem, we recall the norms $||\cdot ||_{A}$ defined on nc power series as

\[\Big|\Big|\sum_{I} a_{I}X_{I}\Big|\Big|_{A}=\sum_{I}|a_{I}|A^{|I|}\]
where I ranges over multi-indices, and $|I|$ is the length of $I$ (see \cite{GMS2}).

\begin{theorem} \label{even-condition}
There exist a $C$ and an $\epsilon$ such that, if $W(X_{1},...,X_{n})$ is a self adjoint nc power series containing only terms of even degree, and $||W||_{C}<\epsilon$, then there is a corresponding power series $V(Y_{1},...,Y_{n})$ such that, when $Y$ has the free Gibbs law associated to $\tfrac{1}{2}|Y|^2+V$, then $Y+\mathcal{D}_{Y}V(Y)$ has the free Gibbs law associated to $\tfrac{1}{2}|X|^{2}+W$. 
\end{theorem}

This is precisely the condition that the free Gibbs law for $\tfrac{1}{2}|Y|^{2}+V(Y)$ pushes forward to that of $\tfrac{1}{2}X^{2}+W(X)$ along $\mathcal{D}(\tfrac{1}{2}|Y|^{2}+V(Y))$. 


In fact, we must take this opportunity to elaborate on the existence of free Gibbs laws. In this perturbative regime, we cannot rely on convexity to ensure the existence of solutions to Schwinger-Dyson, no matter how small the perturbation. Indeed, consider the single variable case and $W=\epsilon X^{3}$.  The functional to minimize in $\tau$ is $\chi(\tau) + \tau(X^2 + W)$.  The value can be reduced by taking any measure which has finite free entropy and translating it left, reducing $\int W$. Since there is no limit to how far we can translate it, and since this effect will eventually overpower the increase in $\int X^{2}$, we find that there can be no minimum. Instead, we must artificially institute a cutoff, requiring that the norm of our random variable is less than $T>2$. Specifically, we invoke a slight modification of (\cite{GMS}):

\begin{proposition} For each cutoff $T>2$, we have that there is an $R>0$ such that $||W||_{T}<R$ implies that there exists a unique solution, $\tau$, to the bounded Schwinger-Dyson equation
\begin{align*}
\tau(P\cdot (X+\mathcal{D}W(X)))=\tau\otimes\tau\times Tr(JP)
\\
|\tau(X_{i_{1}},...,X_{i_{k}})|\leq T^{k}
\end{align*}
\end{proposition}

We will split the proof of Theorem \ref{even-condition} into two main steps\textemdash deriving a differential equation for $V$ in which all terms are cyclic derivatives, and then "integrating" that equation to find a map to which we can apply the contraction mapping theorem to find a solution. Following the proof, we will compare the restrictions in this result to those in the commutative case and discuss potential directions for extension. 

The first step is to rephrase the Schwinger-Dyson equation from an integral equation to a differential equation. To do so, it will be useful to define inner products associated to $\tau$:

\begin{align*}
    \langle a,b \rangle_{M} = \tau(a^* b)
\end{align*}
\begin{align*}
    \langle a \otimes b, c \otimes d \rangle_{ M \otimes  M^{op}}
        = \tau(a^* c) \tau(b^* d)
        = \tau \otimes\tau ((a\otimes b)^* c \otimes d)
\end{align*}
\begin{align*}
    \langle A,B \rangle_{M_{n} (M\otimes M^{op})}
        = \tau \otimes \tau \left(\Tr(A^{*}B)\right)   
\end{align*}

We will omit the subscripts if the ambient space can be inferred. Thus the Schwinger-Dyson equation can be written as 
\[\langle\mathcal{D}U, P\rangle =\langle 1,JP\rangle,\ i.e. \]
\[\langle \mathcal{D}U, P\rangle = \langle J^{*}(1),P\rangle,\ i.e.\]
\[\Rightarrow \mathcal{D}U=J^{*}(1)\]

We will also need some additional operators on nc power series, $\mathcal{S}$, $\mathcal{N}$, $\Sigma$ (the inverse of $\mathcal{N}$), and $\Pi$. These are linear operators on power series in $Y$, which act on monomials as follows. The cyclic symmetrization operator, $S$, is given by
\[\mathcal{S}(x_{i_{1}}...x_{i_{n}})=\frac{1}{n}\sum_{j=1}^{n}x_{i_{j}}...x_{i_{n}}x_{i_{1}}x_{i_{j-1}},\]
on constant terms it acts as the identity. The number operator $N$ is given by 
\[\mathcal{N}(x_{i_{1}}...x_{i_{n}})=nx_{i_{1}}...x_{i_{n}},\]
Finally,
\[\Sigma(x_{i_{1}}...x_{i_{n}})=\frac{x_{i_{1}}...x_{i_{n}}}{n},\]
is defined on power series with no constant term and is the inverse of $\mathcal{N}$ on that space. $\Pi$ is the projection onto power series with no constant term.

With these operators defined, we may state the following lemma. 
\begin{lemma} \label{step1-lemma-of-even-condition}
$V$ satisfies the conclusion of Theorem 3.1 if and only if 
\begin{align*}
\mathcal{S}\Pi \Big[ W(Y+\mathcal{D}V) &+ (\mathcal{N}-1)V + \frac{|\mathcal{D}V|^{2}}{2} 
\\
&- (1\otimes\tau+\tau\otimes 1)\Tr(\log(1+J\mathcal{D}V))\Big]=0
\end{align*}
\end{lemma}

\begin{proof}
Our aim is to express the Schwinger-Dyson equation of the pushforward as a single cyclic derivative. For the purpose of keeping our derivatives clear, we will define the variable $X=Y+\mathcal{D}V(Y)$. We then have that 
\begin{equation}
Y+\mathcal{D}_{Y}V(Y)=J_{Y}^{*}(1) \label{eq:Y}
\end{equation}
and want to understand what condition on $V$ ensures Schwinger-Dyson for $X$, i.e.
\begin{equation}
X+\mathcal{D}_{X}W(X)=J_{X}^{*}(1) \label{eq:X}.
\end{equation}

Substituting the definition of $X$ into (\ref{eq:X}) gives
\begin{align*}
Y + \mathcal{D}_{Y} V(Y) + \mathcal{D}_{X} W(Y + \mathcal{D}_{Y} V(Y)) = J^*_X (1),
\end{align*}
to which we apply the chain rule found in \cite[Lemma 3.1]{dima},
\begin{align*}
    J^{*}_{X}(1)=J^{*}_{Y}\left(\frac{1}{1+J_{Y}\mathcal{D}_{Y}V(Y)}\right),
\end{align*}
to arrive at the equation
\begin{equation} \label{intermediate-equation}
Y+\mathcal{D}_{Y}V(Y)+\mathcal{D}_{X}W(Y+\mathcal{D}_{Y}V(Y))=J^{*}_{Y}\left(\frac{1}{1+J_{Y}\mathcal{D}_{Y}V(Y)}\right).
\end{equation}

Similarly, we apply the chain rule for the cyclic derivative:
\begin{align*}
    \mathcal{D}_Y = (1 + J_Y \mathcal{D}_Y V)\mathcal{D}_X,
\end{align*}
obtaining
\begin{multline}
    Y+\mathcal{D}_{Y}V(Y)+(1+J_{Y}\mathcal{D}_{Y}V)^{-1}\mathcal{D}_{Y}W(Y+\mathcal{D}_{Y}V(Y))\\
    =J_{Y}^{*}\left(\frac{1}{1+J_{Y}\mathcal{D}_{Y}V(Y)}\right). \label{eq:allY}
\end{multline}

In this equation, $1$ is the identity matrix in $M_{n}(M\otimes M^{op})$, the $n \times n$ matrix with $1\otimes 1$ in all its diagonal entries. We know that $1+J_{Y}\mathcal{D}_{Y}V$ is invertible in this space provided that $J_{Y}\mathcal{D}_{Y}V$ has norm less than 1. In our next step, we will be restricting $V$ to a smaller set still, so invertibility is guaranteed.

As $X$ has been removed from our equation and all derivatives are with respect to $Y$ now, we will assume this going forwards and neglect the subscripts.  We expand the right hand side of (\ref{intermediate-equation}) as 
\begin{align*}
J^*\left(\frac{1}{1 + J \mathcal{D} V}\right)
    &= J^* (1) - J^* \left(\frac{J \mathcal{D} V}{1 + J \mathcal{D} V}\right) 
\\
    &= Y + \mathcal{D} V - J^* \left(\frac{J \mathcal{D} V}{1 + J\mathcal{D} V}\right),
\end{align*}
Performing the resulting cancellation and multiplying (\ref{eq:allY}) by $(1+J\mathcal{D}V)$ gives
\begin{align}
    \mathcal{D}W(Y+\mathcal{D}V)=-(1+J\mathcal{D}V)J^{*}\left(\frac{J\mathcal{D}V}{1+J\mathcal{D}V}\right)\label{eq:gettin}
\end{align}

We will expand the right hand side of (\ref{eq:gettin}) and then simplify with the following identity from \cite[Lemma 3.4]{dima}:
\begin{multline}
    \frac{1}{m+1}\mathcal{D}\left[(\tau\otimes 1+1\otimes \tau) \Tr(Jf^{m+1})\right]=-J^{*}(Jf^{m+1})+JfJ^{*}(Jf^{m}).
\end{multline}

Expanding the right hand side of (\ref{eq:gettin}) yields
\begin{align*}
\sum_{n=1}^{\infty} & (-1)^{n} J^* (J \mathcal{D} V^n) + (-1)^{n} J \mathcal{D} V J^* (J \mathcal{D} V^{n}) 
\\
&= -J^{*}(J\mathcal{D}V) + \sum_{n=1}^{\infty}(-1)^{n}(J\mathcal{D}VJ^{*}(J\mathcal{D}V^{n})-J^{*}(J\mathcal{D}V^{n+1}))
\\
&= -J\mathcal{D}VJ^{*}(1)+\mathcal{D}\left[(1\otimes\tau + \tau\otimes1)\Tr\left(J\mathcal{D}V+\sum_{n=1}^{\infty}\frac{(-1)^{n}}{n+1}J\mathcal{D}V^{n+1}\right)\right]
\\
&= -J\mathcal{D}VJ^{*}(1)+\mathcal{D}\left[(1\otimes\tau+\tau\otimes 1)\Tr(\log(1+J\mathcal{D}V))\right]
\\
&= -J\mathcal{D}V\cdot Y-J\mathcal{D}V\cdot \mathcal{D}V+\mathcal{D}\left[(1\otimes\tau+\tau\otimes 1)\Tr(\log(1+J\mathcal{D}V))\right]
\end{align*}

We're left with
\begin{multline}
    \mathcal{D}(W(Y+\mathcal{D}V)) = - J\mathcal{D}V \cdot Y - J \mathcal{D} V \cdot \mathcal{D} V 
        \\
        + \mathcal{D}\left[(1 \otimes \tau + \tau \otimes 1) \Tr \log (1 + J \mathcal{D} V)\right] \label{all-but-two-terms}
\end{multline}
which nearly expresses the equation as a total (cyclic) derivative.  All that remains is writing the first two terms of the right hand side of (\ref{eq:gettin}) above as cyclic derivatives. Analyzing the remaining two terms of (\ref{eq:gettin}), we make use of the operators defined earlier, noticing
\[Jg\cdot Y=\mathcal{N}g\]
for any $g$. Thus, when $g=\mathcal{D}V$, we get
\[J\mathcal{D}V\cdot Y=\mathcal{N}\mathcal{D}V=\mathcal{D}(\mathcal{N}-1)V\]
We can also see that 
\[J\mathcal{D}V\cdot \mathcal{D}V=\mathcal{D}\left(\frac{\mathcal{D}_{1}V^{2}+\mathcal{D}_{2}V^{2}+...+\mathcal{D}_{n}V^{2}}{2}\right)=\mathcal{D}\left(\frac{|\mathcal{D}V|^{2}}{2}\right).\]

So equation (\ref{eq:gettin}) can be rewritten as 
\begin{multline}
\mathcal{D}\left[W(Y+\mathcal{D}V)+(\mathcal{N}-1)V+\frac{|\mathcal{D}V|^{2}}{2}\right.\\-(1\otimes\tau+\tau\otimes 1)\Tr(\log(1+J\mathcal{D}V))\Big]=0\label{eq:BigD}
\end{multline}

Since $\mathcal{D}$ only sees the cyclically symmetric part of power series, and does not see constants, this is equivalent to the desired equation 
\begin{align*}
\mathcal{S}\Pi \Big[ W(Y+\mathcal{D}V) &+ (\mathcal{N}-1)V + \frac{|\mathcal{D}V|^{2}}{2} 
\\
&- (1\otimes\tau+\tau\otimes 1)\Tr(\log(1+J\mathcal{D}V))\Big]=0
\end{align*}

Thus concludes this lemma as well as the first step in the proof of Theorem \ref{even-condition}, deriving a differential equation for $V$ in which all terms are cyclic derivatives.
\end{proof}

We proceed to the second step in the proof of Theorem \ref{even-condition}, where we "integrate" the above equation to find a map to which we can apply the contraction mapping theorem in order to find a solution.

We rephrase the differential equation in Lemma \ref{step1-lemma-of-even-condition}:
\begin{multline}
\mathcal{S}\Pi\mathcal{N} V= \mathcal{S}\Pi\Big[-W(Y+\mathcal{D}V) + V - \frac{|\mathcal{D}V|^{2}}{2} 
\\
+ (1\otimes\tau+\tau\otimes 1)\Tr(\log(1+J\mathcal{D}V))\Big]. \label{eq:Picard}
\end{multline}
It will be more useful to solve for $\tilde{V}=S\Pi\mathcal{N}V$, which must satisfy 
\begin{multline}\tilde{V}=\mathcal{S}\Pi\\\left[-W(Y+\mathcal{D}\Sigma \tilde{V})+\Sigma \tilde{V}-\frac{|\mathcal{D}\Sigma \tilde{V}|^{2}}{2}+(1\otimes\tau+\tau\otimes 1)\Tr(\log(1+J\mathcal{D}\Sigma \tilde{V}))\right]\label{eq:Picard2}
\end{multline}

Whenever necessary, we will denote the right hand side by $F(\Sigma \tilde{V})$. We will show that there is a set on which $F(\Sigma \cdot)$ is a contraction. Along the way, we must prove two lemmas.
\begin{lemma}
$F(\Sigma\cdot )$ preserves evenness of power series. In other words, if $U$ has only terms of even degree, then $F(\Sigma U)$ also has only terms of even degree. 
\end{lemma}

In turn, proving this requires an easy proposition:
\begin{proposition} If $U$ is a potential which contains only even terms, then $\tau_{U}(P)=0$ for any polynomial $P$ which contains only odd terms. 
\end{proposition}
\begin{proof}
This is a corollary of uniqueness of free Gibbs measures, \cite{Gui06}. In particular, if $X$ has free Gibbs law $\tau_{U}$, then $Y=-X$ also satisfies
\[\tau(P(Y)\cdot \mathcal{D}_{Y}U(Y))=-\tau(P(-X)\cdot \mathcal{D}_{X}U(X))\]
\[=-\tau\otimes\tau(\Tr(J_{X}P(-X)))=\tau\otimes\tau (\Tr(J_{Y}P(Y)))\]
So by uniqueness, $-X$ has the same law as $X$, yet $\tau(P(X))=-\tau(P(-X))$ for any odd polynomial, so this must be zero.
\end{proof}

\begin{proof}[\proofname\ of Lemma 3.4]
We must check that each term preserves evenness. The term $W(Y+\mathcal{D}\Sigma V)$ certainly does, since all terms in $W$ are even, and all terms in $Y+\mathcal{D}\Sigma V$ are odd. The term $\Sigma V$ is most immediate of all, and every term in $|\mathcal{D}\Sigma V|^{2}$ is a product of two odd factors. To see that the log term also preserves this, we expand it into its Taylor series 

\[\sum_{n}\frac{(-1)^{n}}{n}(1\otimes \tau+\tau\otimes 1)\Tr((J\mathcal{D}\Sigma V)^{n})\]

Considering now a fixed $n$, we see that each term in $J\mathcal{D}\Sigma V$ is of the form $a\otimes b$ where the degrees of $a$ and $b$ sum to an even number. The same is thus true of all powers. If $(1\otimes \tau)$ or $(\tau \otimes 1)$ were to produce a term with odd degree, it would be multiplied by $\tau(a)$ where $a$ also had odd degree, and so, by the proposition, would be zero.
\end{proof}

We note that the $\log(1+\mathcal{J}\mathcal{D}\Sigma\tilde{V})$ and  $W(Y+\mathcal{D}\Sigma V)$ terms produce the requirement that $W$ be even.  If $V$ contains any terms of odd degree, both of these terms can produce linear (degree one) terms in $F(\Sigma V)$ on which $\Sigma V$ is not strictly contractive.

Next, we introduce the sets that we will consider as domains for $F(\Sigma\cdot)$: $E\cap B_{A,R}$ where $E$ is the space of nc power series with only even, positive degree terms, and $B_{A,R}$ is the ball of $||\cdot ||_{A}$ radius $R$. The previous lemma shows that $F(\Sigma\cdot)$ preserves $E$. We also need
\begin{lemma}
If $A\geq 1$, $F(\Sigma\cdot)$ has a Lipschitz constant on $B_{A,R}\cap E$ bounded above by
\[\frac{1}{2}+\left|\left|\sum_{i}\partial_{i}W\right|\right|_{B\otimes B}+R+\frac{4R}{A^{2}-2R}\]
where 
\[||\sum_{I,J}a_{I}b_{J}X_{I}\otimes X_{J}||_{A\otimes B}=\sum_{I,J}|a_{I}||b_{j}|A^{|I|}B^{|I|}\]
and, in the above bound, $B=A+R$. 

Moreover
\[||F(\Sigma\cdot)||_{A} \leq ||W||_{B}+||V||_{A}\left(\tfrac{1}{2}+R+\frac{4R}{A^{2}-2R}\right)\]

\end{lemma}
\begin{proof}

Most of this proof can be reduced to an appeal to Cor. 3.12 in \cite{GS}. However, two terms deserve a comment.

Unlike $\cite{GS}$, our $F$ contains a $\Sigma V$:
\[||\Sigma V-\Sigma U||_{A}\leq \frac{1}{2} ||V-U||_{A}\]
Which follows immediately from the fact that all terms in $U$ and $V$ are of order 2 or greater. 

Additionally, our bound for the log term is different from that in $\cite{GS}$ so we briefly comment on it's proof. We begin by taylor expanding the log term as
\[\sum_{n}\frac{(-1)^{n}}{n}(1\otimes \tau+\tau\otimes 1)\Tr\left((J\mathcal{D}\Sigma V)^{n}-(J\mathcal{D}\Sigma U)^{n}\right),\]
a notationally tedious, but otherwise straightforward calculation then shows that this is bounded in norm by 
\[\sum_{n}\frac{2^{n+1}R^{n}}{A^{2n}}||V-U||_{A}=||V-U||_{A}\frac{4R}{A^{2}-2R},\]
see \cite[Lemma 3.8]{dima} for more detail.

We can then obtain the desired bounds almost immediately from the Lipschitz constants, with the only exception being the $W$ term, for which we use the bound
\[||W(Y+\mathcal{D}\Sigma V)||_{A}\leq ||W||_{B}\]
which follows from $\max(||Y_{i}+\mathcal{D}_{i}\Sigma V||_{A})\leq B$.
\end{proof}

We are now equipped to prove Theorem \ref{even-condition}.

\begin{proof}[\proofname\ of Theorem \ref{even-condition}]
We fix a cutoff $3\geq T>2$. We then choose $A=3$ and an $R<1/4$ so that $||V||_{A}<R$ implies the existence of a unique free Gibbs law with support bounded by $T$.  Then we find that the Lipschitz constant of $F(\Sigma\cdot)$ is bounded by:
\[\left|\left|\sum_{i}\partial_{i}W\right|\right|_{\tfrac{13}{4}\otimes \tfrac{13}{4}}+\frac{1}{4}+\frac{1}{9-1/2}\leq ||W||_{\tfrac{17}{4}}+\frac{59}{68}\]
Where we have used that 
\[\left|\left|\sum_{i}\partial_{i}W\right|\right|_{A\otimes A}\leq \sum_{I}|W_{I}||I|A^{|I|-1}\leq \sum_{I}|W_{I}|(1+A)^{|I|}=||W||_{A+1}\]
Moreover, $||V||_{A}<R$ implies that 

\[||F(\Sigma V)||_{A}\leq ||W||_{\tfrac{13}{4}}+R\left(\frac{1}{2}+R+\frac{4R}{9-2R}\right)\]
\[\leq ||W||_{\tfrac{13}{4}}+\frac{59}{68}R\]

Since 
\[||W||_{\tfrac{13}{4}}\leq ||W||_{\tfrac{17}{4}}\]

We find that if $||W||_{\tfrac{17}{4}}<\frac{9}{68}$ then $F(\Sigma \cdot)$ will be a contraction, and if $||W||_{\tfrac{17}{4}} < \frac{9}{68}R$, then we can also be assured that it will map $E\cap B_{A,R}$ into itself, so we have a fixed point $V$. We then find that $\tilde{V}=\Sigma V$ satisfies the conclusion of the theorem, since $\Sigma$ can only decrease $||\cdot||_{A}$. Thus, we take $C=\tfrac{17}{4}$ and $\epsilon=\tfrac{9}{68}R$.
 
\end{proof}

\section{Open Questions}

There are several interesting questions that remain.  Foremost is whether we can remove the evenness restriction on power series $W$ in Theorem \ref{even-condition} to get a result more in line with the single variable case. More broadly, is it possible to generalize the proof in Section 2 to the multivariable case, which requires more work extending variational techniques to the non-commutative setting.

Second, free probability can in many ways be considered as a model of large random matrices. For example, free Gibbs laws are the limit of corresponding Gibbs laws on large random matrices. Is this also true for moment measures and their potentials? In particular, if we apply the classical result to large matrix models, does the result converge (in law or otherwise) to the free Gibbs law corresponding to the free moment measure. 

We have managed to reduce this question to the following.  Let $X = (\mathcal{D}U) Y$ where $X$ has free Gibbs law $\tau_T$ and $Y$ has free Gibbs law $\tau_U$, and construct $\nu_N$ measures on $\R^{n N^2}$ as the matrix model for the law $\tau_T$  and measures $\mu_N$ for the matrix model for $\tau_U$.  Let $\rho_N = e^{-U_N} \,dx$ where $U_N$ is the classical moment measure potential for $\nu_N$.  The problem can be reduced to the question of whether $\tfrac{1}{N} W_2((\mathcal{D}U)_{\#} \rho_N, \nu_N) \to 0$, where $W_2$ is the classical Wasserstein distance between these two measures.  

This in turn raises a general question.  Does the classical Wasserstein distance between two sequences of measures for matrix models of nc laws converge to the (noncommutative) Wasserstein distance (see \cite{BV-Wasserstein}) between the nc laws themselves?  What about the (smaller) noncommutative Wasserstein distance between the matrix laws?

Note that a similar statement, convergence in nc law implies convergence in Wasserstein distance, is not true in general (and can be seen as a consequence of the falsehood of the Connes embedding conjecture).  This is in contrast to the classical case, where Wasserstein convergence and convergence in law are equivalent for random variables with a uniform bound.

Finally, there is the question of whether this and related works can be extended to a broader class of laws than just free Gibbs laws with power series potentials.

\appendix
\section{Computing Free Gibbs Laws for Single Variables}

This section is intended as a quick overview to methods for solving the equation $2\pi H(\rho)=u'$ among measures on $\mathbb{R}$.  We consider the Cauchy transform
\[G_{\rho}(z)=\int \frac{1}{z-t}d\rho\]
which in particular satisfies 
\[\lim_{y\downarrow 0}G_{\rho}(x+iy)=\pi (H(\rho)-i\rho)(x)\]

We would like to find $G$ using the fact that its real part is known, but we only know this real part on the support of $\rho$ (which is also, a priori, unknown). This is remedied by noticing that $G$ is an analytic function on the Riemann sphere minus the support of $\rho$. With convex potentials, the support of $\rho$ is connected, so we may assume that $G$ is an analytic function away from some compact subinterval of $\mathbb{R}$. For the sake of brevity, we will assume that the potential is even, so the measure is supported on a symmetric interval $[-r,r]$. We then try to find the Cauchy transform as 
\[G(z)=F(R(z))\]
for some $F$, holomorphic on the interior of the disk, and 
\[R(z)=\frac{\sqrt{z^{2}-r^{2}}-z}{r}\]
the Riemann mapping from $S^{2}\setminus [-r,r]$ to the disk. We make note of the inverse of this map:
\[S(w)=-\frac{r(1+w^{2})}{2w}=\frac{r\left(\frac{w+1}{1-w}\right)^{2}+r}{1-\left(\frac{w+1}{1-w}\right)^{2}}\]
The defining equation of $\rho$ now gives that 
\[\lim_{z\rightarrow e^{i\theta}}F(z)=\tfrac{1}{2}u'(-r\cos(\theta))-i\pi\sgn(\sin(\theta))\rho(-r\cos(\theta))\]
In particular, the real part is enough to compute the Taylor series for $F$; if $F=\sum a_{n}z^{n}$, then 
\[a_{n}=\frac{1}{\pi}\int u'(-rcos(\theta))e^{-in\theta}d\theta\]

What remains is to fix $r$; we consider a contour $\gamma_{\epsilon}$ which traces the rectangle with sides $re(z)=\pm r$ and $im(z)=\pm \epsilon$, oriented clockwise. Since we know the limit of $G$ as $z$ approaches the axis and that $\rho$ is a probability measure, we can see on the one hand that 
\[\int_{\gamma_{\epsilon}}G(z)dz=-2\pi i\]
but on the other that
\[\int_{\gamma_{\epsilon}} G(z) dz = \int_{\gamma_{\epsilon}} F(R(z)) dz\rightarrow \int_{S^1} F(w) S'(w) dw \]
\[= \frac{r}{2} \int_{S^{1}} F(w) \left(\frac{1}{w^2} - 2w\right) = r\pi i a_{1} \]
Whence,
\[r a_1 = -2\]

We illustrate the process with the potential $u=x^4/4$ from section 2.5. We have that 
\[\textrm{Re}(F(e^{i\theta}))=-\tfrac{r^{3}}{2}\cos^{3}(\theta)=-\tfrac{r^{3}}{8}(\cos(3\theta)+3\cos(\theta))\]
so $a_{1}=-\tfrac{3}{8}r^{3}$ and $a_{3}=-\tfrac{1}{8}r^{3}$, and all other Taylor coefficients are zero. We then fix $r$ using the equation
\[ra_{1}=-2\Rightarrow r^{4}=\tfrac{16}{3}\]

Then we see that 
\[-\pi*\nu(-r\cos(\theta))=\textrm{Im}(F(e^{i\theta}))=-\tfrac{r^{3}}{8}(\sin(3\theta)+3\sin(\theta))\]

so
\[\nu(x)=\tfrac{r^{3}}{8\pi}\left(4\frac{x^{2}}{r^{2}}+2\right)\sqrt{1-\frac{x^{2}}{r^{2}}}.\]




\bibliography{references.bib}

\end{document}